\documentclass[12pt]{article}
\usepackage{amsmath,amssymb,amsthm,enumitem,graphics,tikz}
\usepackage{mathrsfs}
\usepackage{amssymb,amsfonts,amsmath}

\newtheorem{theorem}{Theorem}
\newtheorem{lemma}{Lemma}
\begin{document}

\title{Proper Orientation Number of Triangle-free Bridgeless Outerplanar Graphs }

\author{J. Ai$^{1}$, S. Gerke$^{2}$, G. Gutin$^{1}$,  Y. Shi$^{3}$ and Z. Taoqiu$^{3}$\qquad\\
\small $^{1}$ Department of Computer Science\\ \small Royal Holloway, University of London\\ \small Egham, Surrey, TW20 0EX, UK\\ \small Jiangdong.Ai.2018@live.rhul.ac.uk, g.gutin@rhul.ac.uk\\
\small $^{2}$ Department of Mathematics\\ \small Royal Holloway, University of London\\ \small Egham, Surrey, TW20 0EX, UK\\ \small Stefanie.Gerke@rhul.ac.uk\\
\small $^{3}$ Center for Combinatorics and LPMC\\ \small Nankai University, Tianjin, P.R. China 300071\\ \small shi@nankai.edu.cn, tochy@mail.nankai.edu.cn}
\maketitle

\begin{abstract}
An orientation of $G$ is a digraph obtained from $G$ by replacing each edge by exactly one of two possible arcs with the same endpoints. We call an orientation \emph{proper} if neighbouring vertices have different in-degrees. The proper orientation number of a graph $G$, denoted by $\vec{\chi}(G)$, is the minimum maximum in-degree of a proper orientation of G. Araujo et al. (Theor. Comput. Sci. 639 (2016) 14--25) asked whether there is a constant $c$ such that $\vec{\chi}(G)\leq c$ for every outerplanar graph $G$ and showed that $\vec{\chi}(G)\leq 7$ for every cactus $G.$ We prove that $\vec{\chi}(G)\leq 3$ if $G$ is a triangle-free $2$-connected outerplanar graph and $\vec{\chi}(G)\leq 4$ if $G$ is a triangle-free bridgeless outerplanar graph. 
\vspace{0.3cm}\\
{\bf Keywords:} proper orientation; proper orientation number; outerplanar graph.

\end{abstract}

\section{Introduction}


In this paper we consider graphs with no loops or multiple edges. An embedding of a planar graph into the plane is {\em outerplanar} if all vertices of $G$ belong to some face of $G'.$
A planar graph $G$ is {\em outerplanar} if it has an outerplanar embedding.
Outerplanar graphs are not only of interest in graph theory, they are also useful in applications, see, e.g., Gutin et al. \cite{GMS}, where it is proved that  the so-called vertex horizontal graphs used in the analysis of time series  are precisely outerplanar graphs with a Hamilton path.

An orientation  of a graph $G=(V,E)$ is a digraph $D=(V,A)$ obtained from $G$ by replacing each edge by exactly one of two possible arcs with the same end-vertices. 
For each $v \in V$, the {\em in-degree} of $v$ in $D=(V,A)$, denoted by $d^-_D(v)$, is the number of vertices $u$ such that $(u,v)\in A.$ 
An orientation $D$ of $G$ is proper if $d^-_D(v)\ne d^-_D(u)$, for all $uv\in E(G)$. An orientation with maximum in-degree at most $k$ is called a {\em proper $k$-orientation}. The {\em proper orientation number} of a graph $G$, denoted by $\vec{\chi}(G)$, is the minimum integer $k$ such that $G$ admits a proper $k$-orientation.

This graph parameter was introduced by Ahadi and Dehghan \cite{AD}. They observed that this parameter is well-defined for any graph $G$ since one can always obtain a proper $\Delta(G)$-orientation, where $\Delta(G)$ is the maximum degree of $G$. This fact can be proved by induction on the size of $G$ by removing a vertex of maximum degree and orienting all edges towards  this vertex. Every proper orientation of a graph $G$ induces a proper vertex-coloring of $G$. Thus,
$\omega(G)-1 \leq \chi(G)-1 \leq \vec{\chi}(G) \leq \Delta(G),$
where $\omega(G)$ is the number of vertices in a maximum clique of $G$. 
Ahadi and Dehghan \cite{AD} proved that it is NP-complete to compute $\vec{\chi}(G)$ even for planar graphs.  Araujo et al. \cite{AN} strengthened this result by showing it holds for bipartite planar graphs of maximum degree $5$. 

A parameter $\alpha$ is \emph{monotonic} if $\alpha(H)\leq \alpha(G)$ for every induced subgraph $H$ of $G$. Unfortunately, the proper orientation number is not monotonic; an example of a tree $T$ and its leaf $x$ such as $\vec{\chi}(T)=2$ but $\vec{\chi}(T-x)=3$ is given in  \cite{AJ}. This makes it difficult to prove upper bounds on the parameter even for relatively narrow classes of graphs. Araujo et al. \cite{AJ} asked whether there is a constant $c$ such that $\vec{\chi}(G)\leq c$ for every outerplanar graph $G$. They proved that for any cactus $G$ (where every 2-connected component is either an edge or a cycle), we have $\vec{\chi}(G)\leq 7$ and  that for any tree $T,$ $\vec{\chi}(T)\leq 4$ (see also \cite{KM} for a short algorithmic proof). 

We prove the following results.

\begin{theorem}\label{main2}
For any triangle-free, $2$-connected, outerplanar graph $G$, we have $\vec{\chi}(G)\leq 3$. The bound is tight.
\end{theorem}

Note that the intersection between cacti and $2$-connected outerplanar graph consists only of cycles. The proof of Theorem~\ref{main2} is non-trivial and requires a good understanding of difficult cases in bounding the parameter for a triangle-free $2$-connected outerplanar graph $G$ when, using an outerplanar embedding of $G$, $G$ is constructed from a cycle by adding ears one by one. In essence, the proof consists of an appropriate algorithm and its correctness proof. 
The algorithm can be run in linear time.

By increasing the bound by one, we can widen the class of outerplanar graphs as follows. The proofs are based on Theorem \ref{main2}. A {\em bridge} is an edge whose deletion increases the number of connected components of $G.$ A graph is {\em bridgeless} if it has no bridges. We call a graph \emph{tree-free} if after removing all bridges the 2-connected components have size at least $3.$

\begin{theorem}\label{bridgefree}
For any triangle-free, bridgeless, outerplanar graph $G$, we have $\vec{\chi}(G)\leq 4$. 
\end{theorem}


\begin{theorem}\label{treefree}
For any triangle-free, tree-free, outerplanar graph $G$, we have $\vec{\chi}(G)\leq 4$. 
\end{theorem}

In the literature, there are some other interesting upper bounds for the parameter on planar graphs. In particular,   Knox el al. \cite{KM} proved that $\vec{\chi}(G)\leq 5$ for a
$3$-connected planar bipartite graph $G$ and Noguci \cite{KN} showed that $\vec{\chi}(G)\leq 3$ for any bipartite planar graph with $\delta(G) \geq 3.$

This paper is organized as follows. In Section~\ref{sec:path}, we give some simple lemmas for proper orientations of a path. In Section~\ref{2connected}, we construct  triangle-free $2$-connected outerplanar graphs $G$ by choosing a finite face and adding to it finite faces one by one such that every newly added face shares precisely one edge; we can talk about adding paths instead of faces. Then we give an algorithm to obtain a proper $3$-orientation of $G$ by orienting each added path. In fact, the algorithm has to foresee what further paths will be added, in particular, by paying attention to an especially tricky situation of eventually constructing a $k$-fan (its definition is also given in Section~\ref{2connected}). We conclude this section by giving an example of a $2$-connected outerplanar graph $G$ with $\vec{\chi}(G)=3$.
In Section~\ref{sec:bridgefree}, we provide proofs of Theorems~\ref{bridgefree} and \ref{treefree} based on Theorem \ref{main2}. We conclude the paper in Section \ref{sec:prob}, where we discuss some open problems. 

\section{Orientations on paths}\label{sec:path}

In this section, we collect some simple lemmas about proper orientations on paths.  For a path $P$ we denote by $||P||$ the number of edges of $P$.

\begin{lemma} \label{path}
Let $P=v_1\dots v_n$ be a path.
\begin{enumerate}
\item \label{path6}  If $||P||\geq 5$ then there exists a proper orientation such that $d^{-}(v_1)=0$, $d^{-}(v_2)=1$, $d^{-}(v_{n-2})=0$, $d^{-}(v_{n-1})=2$ and $d^{-}(v_n)=0.$ 
\item \label{path5} If $||P||=4$ then there are two proper orientations such that $d^{-}(v_1)=0$ and $d^{-}(v_5)=0$ and 
\begin{enumerate}
\item \label{path5_1} $d^{-}(v_1)=0,$ $d^{-}(v_2)=1,$ $d^{-}(v_3)=2,$ $d^{-}(v_4)=1,$ $d^{-}(v_5)=0$, and 
\item \label{path5_2} $d^{-}(v_1)=0,$ $d^{-}(v_2)=2,$ $d^{-}(v_3)=0,$ $d^{-}(v_4)=2,$ $d^{-}(v_5)=0,$ respectively.
\end{enumerate}
\item \label{path4} If $||P||=3$ then there exists a proper orientation with  $d^{-}(v_1)=0,$ $d^{-}(v_2)=1,$ $d^{-}(v_3)=2$ and $d^{-}(v_4)=0.$
\end{enumerate}
\end{lemma}

\begin{proof}
	
\begin{enumerate}
\item
First, let us remark that every vertex of in-degree $0$ in a proper orientation of a path with $n$ vertices can be replaced by 
\includegraphics[scale=0.25]{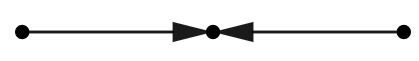}
to obtain a proper orientation of a path with $n+2$ vertices. So it suffices to find an orientation of a path of length $5$ and $6$ as we can then extend these paths by repeatedly replacing $v_{n-2}$ as above. These orientations can be easily found:

\begin{center}
	\includegraphics[scale=0.25]{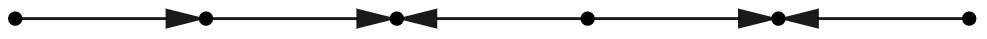}\\
	\includegraphics[scale=0.25]{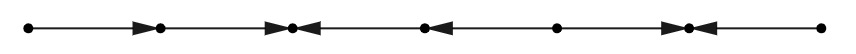}
\end{center}

\item 


\includegraphics[scale=0.25]{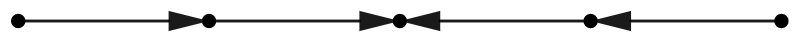}\\
\includegraphics[scale=0.25]{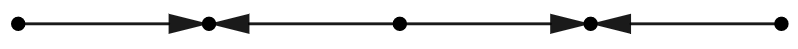}

\item
%
\includegraphics[scale=0.25]{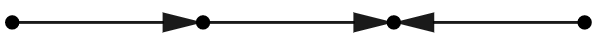}

\end{enumerate}
\end{proof}

Lemma~\ref{path} is simple, but very helpful. Let $G$ be a graph with a proper orientation. If we are given a path $P$ with $||P||\not=4$, then we can add $P$ to any of the edges of $G$ and using the orientations of Lemma~\ref{path} to obtain a proper orientation of the new graph. To see this, consider an edge $e=\{a,b\}$ and 
 a proper orientation $D$ of $G$. 
If $d_D^{-}(a)=2$, then we use the orientation of $P$ such that $a=v_1$ and $b=v_n$. As $D$ is a proper orientation, we know that $d_D^{-}(b)\not= 2$ and therefore the orientation of the new graph is proper. 
(In what follows, we omit subscripts in in-degrees when the orientation is clear from the context.)
If $d^{-}(a)=1$, then we use the orientation of $P$ such that $a=v_n$ and $b=v_1$. As before, $D$ is a proper orientation so $d^{-}(b)\not= 1$ and therefore the orientation of the new graph is proper. Now we can reverse the roles of $a$ and $b$. Finally if $d^{-}(a), d^{-}(b)\not\in \{1,2\}$, then we can just set $a=v_1$ and $b=v_n$ (or the other way round). Note that the orientations of the edges in $G$ and the in-degrees of any vertex in $G$ do not change.

If $||P||=4$, then the situation becomes more complicated if $d^{-}(a)=1$ and $d^{-}(b)=2$. (In all other cases we just choose the orientation of Lemma~\ref{path} that avoids the in-degrees.) To deal with this case, we need other proper orientations of graphs. Some orientations are repeated for easier reference later.

\begin{lemma}\label{connectfan}
Let $P=v_1\dots v_n$ be a path with $||P||\geq 4$.  Then there exists a proper orientation such that $d^{-}(v_1)=0$, $d^{-}(v_2)=2$, $d^{-}(v_{n-1})=2$ and $d^{-}(v_n)=0.$ 
 \end{lemma}
 
 \begin{proof}
 As in the proof of Lemma~\ref{path}, we can replace in any path with a proper orientation a vertex of in-degree $0$ by a path of length $2$ with both edges oriented inwards to obtain a proper orientation for a path with two more edges. Observe that the second orientation of a path of length $4$ satisfies the conditions of this lemma, it suffices to give an orientation of a path of length $5$.
%
\begin{center}
	\includegraphics[scale=0.25]{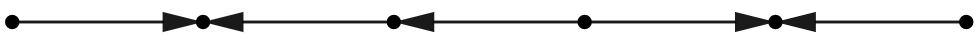}
\end{center}

\end{proof}

\begin{lemma}\label{in1}
Let $P=v_1\dots v_n$ be a path with $||P||\geq 4$. Then there exists a proper orientation such that $d^{-}(v_1)=1$, $d^{-}(v_2)=0$, $d^{-}(v_{n-1})=0$ and $d^{-}(v_n)=1.$ 
\end{lemma}

\begin{proof}
As in the proof of Lemma~\ref{path}, we can replace in any path with a proper orientation a vertex of in-degree $0$ by a path of length $2$ with both edges oriented inwards to obtain a proper orientation for a path with two more edges. So we need to exhibit orientations for the paths of length $4$ and $5$.

\begin{center}
	\includegraphics[scale=0.25]{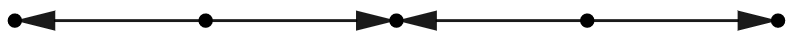}\\
	\includegraphics[scale=0.25]{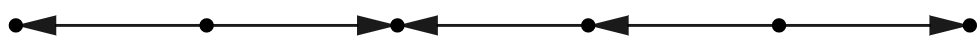}
\end{center}
%
%
%
%
%
%
%
%

\end{proof}
 
 \begin{lemma}\label{connectfan23}
Let $P=v_1\dots v_n$ be a path.
 \begin{enumerate}
\item  If $||P||\geq 5$ and $||P||$ is odd, then there exists a proper orientation such that $d^{-}(v_1)=0$, $d^{-}(v_2)=2$, $d^{-}(v_{n-2})=0$, $d^{-}(v_{n-1})=2$ and $d^{-}(v_n)=0$.

\item If $||P||\geq 6$ and $||P||$ is even, then there exists a proper orientation such that $d^{-}(v_1)=0$, $d^{-}(v_2)=2$, $d^{-}(v_{n-1})=1$ and $d^{-}(v_n)=0$.

\item If $||P||=4$, then there exists a proper orientation such that $d^{-}(v_1)=0$, $d^{-}(v_2)=2$, $d^{-}(v_{3})=d^{-}(v_{n-2})=0$, $d^{-}(v_{4})=d^{-}(v_{n-1})=2$ and $d^{-}(v_{5})=d^{-}(v_n)=0$. 

\item If $||P||=3$, then there exists a proper orientation such that $d^{-}(v_1)=0$, $d^{-}(v_2)=2$, $d^{-}(v_3)=1$, $d^{-}(v_4)=0$.
\end{enumerate}

\end{lemma}
\begin{proof}
As in the proof of Lemma~\ref{path}, we can replace in any path with a proper orientation a vertex of in-degree $0$ by a path of length $2$ with both edges oriented inwards to obtain a proper orientation for a path with two more edges. So we only have to consider the paths of length $5$ and $6$ respectively.
\begin{enumerate}
\item  This is the same proper orientation on a path of length $5$ as given in the proof of Lemma~\ref{connectfan}.
\item  
%
\includegraphics[scale=0.25]{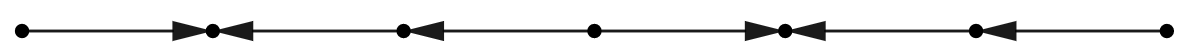}

\item This is the same statement as in Lemma~\ref{path} Case \ref{path5_2}.
\item This is the same statement about path of length $3$ as in Lemma~\ref{path} just in reverse order.
\end{enumerate}
\end{proof}

 \begin{lemma}\label{connectfan23a}
 Let $P=v_1\dots v_n$ be a path.
 \begin{enumerate}
\item If $||P||\geq 5$ and $||P||$ is odd, then there exists a proper orientation such that $d^{-}(v_1)=0$, $d^{-}(v_2)=1$, $d^{-}(v_{n-2})=0$, $d^{-}(v_{n-1})=2$ and $d^{-}(v_n)=0$.
\item If $||P||\geq 4$ and $||P||$ is even, then there exists a proper orientation such that $d^{-}(v_1)=0$, $d^{-}(v_2)=1$, $d^{-}(v_{n-1})=1$ and $d^{-}(v_n)=0$.
\end{enumerate}
\end{lemma}

\begin{proof}
For $||P||=4$, the orientation in Lemma~\ref{path} Case \ref{path5_1} satisfies the condition. For the remaining cases
as in the proof of Lemma~\ref{path} we can replace in any path with a proper orientation a vertex of in-degree $0$ by a path of length $2$ with both edges oriented inwards to obtain a proper orientation for a path with two more edges. So we only have to consider the paths of length $5$ and $6$ respectively.
\begin{enumerate}
\item 
%
\includegraphics[scale=0.25]{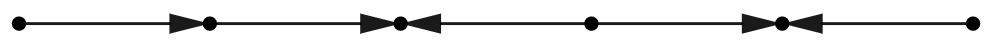}

\item
%
\includegraphics[scale=0.25]{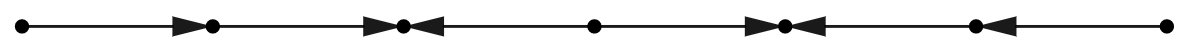}

\end{enumerate}
\end{proof}

\section{Proof of Theorem \ref{main2}}\label{2connected}

\subsection{Description of the algorithm}

Let $G$ be a 2-connected, triangle-free, outerplanar graph. Consider an outerplanar embedding $G'$ of $G$ into the plane such that all vertices of $G$ belong to the infinite face of $G'.$ We can construct $G$ by choosing a finite face (which is a chordless cycle in $G$) and adding to it finite faces one by one such that every newly added face shares a vertex with one of the previously added faces. Since $G$ is $2$-connected, we may assume that every newly added face shares at least an edge with one of the previously added faces. Moreover, the added face shares precisely one edge since $G'$ is an outerplanar embedding. In what follows, we will talk about adding paths, instead of faces, between the vertices of an edge: such a path and the corresponding edge form the boundary of the corresponding face. We will say that such a path is {\em attached} to the corresponding edge.

To describe an algorithm that provides a proper 3-orientation of $G$, we need the following terminology. We call an edge \emph{active} if there will be a path attached to it in the future.  Otherwise we call it \emph{inactive}. Observe that when we add a path some (but maybe not all) of its edges will be active, then they become inactive and never become active again. 
Let $D$ be an orientation of a subgraph $H$ of $G.$ We call a vertex of $D$ with in-degree $2$ adjacent to a vertex of in-degree $3$ a \emph{trouble maker}. An edge of $H$ is called a {\em $1$-$2$ edge} if the in-degrees of its end-vertices in $D$ are $1$ and $2$.

It turns out that active $1$-$2$ edges are problematic when the vertex with in-degree $2$ is a trouble maker. So assume we have an active $1$-$2$ edge $\{a,b\}$ so that $d^{-}(a)=1$, $d^{-}(b)=2$ and $b$ is a trouble maker.
Assume we  attach to it a path of length $3$ oriented as in Lemma~\ref{path} part \ref{path4}. This gives a new active $1$-$2$ edge with $b$ as a trouble maker. We can repeat this $k-1$ times (to get $k$ paths of length $3$ altogether), see the following picture for $k=3$:
\begin{center}
%
%
%
\includegraphics[scale=0.5]{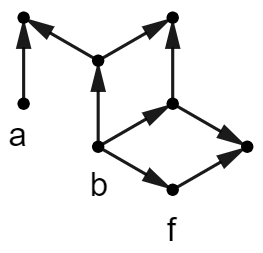}
\end{center}

If we now want to add a path of length $4$ to the active $1$-$2$ edge $\{b,f\}$ without changing the existing orientation we would get a problem. If a path of length $4$ is indeed added we call this structure a \emph{$k$-fan} and use the orientation as in Figure~\ref{normal_kfan}.  Later in our construction, we will always make sure that the path between $a$ and $b$ is oriented in a way that the in-degree of $a$ will be $3$ (that is, both edges incident to $a$ will be oriented towards $a$) and $b$ has in-degree $2$. Note that a $0$-fan is simply a path of length $4$ and that we do not include the original $1$-$2$ edge in our $k$-fan (we include only the $k$ paths of length $3$ and the path of length $4$).
\begin{figure}[h]
\begin{center}
%
%
%
%
%
%
%
%
%
%
%
%
%
%
%
%
\includegraphics[scale=0.5]{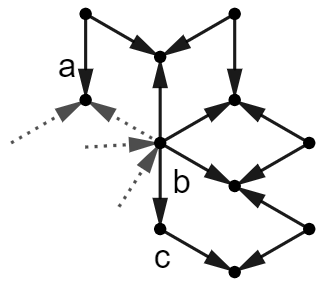}
\end{center}
\caption{\label{normal_kfan} The orientation of a $k$-fan if no $k'$-fan is added to $\{b,c\}$}
\end{figure}

Note that  the orientation in Figure~\ref{normal_kfan} creates a new  $1$-$2$ edge $\{b,c\}$ which may be active.  If another $k'$-fan is added to this edge then we use the orientation in Figure~\ref{added_kfan} where both edges in the path of length $4$ are oriented towards $c$ (otherwise the two orientations are identical). Note that $c$ will have in-degree equal to $3$ when a $k'$-fan is added and all the other in-degrees in the original $k$-fan (including $b$'s) will not change.
\begin{figure}[h]
\begin{center}
%
%
%
%
%
%
%
%
%
%
%
%
%
%
%
%
\includegraphics[scale=0.5]{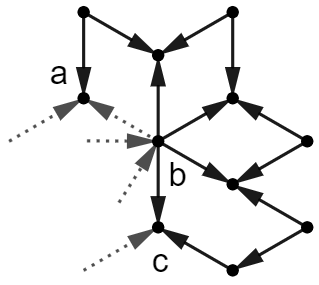}
\end{center}
\caption{\label{added_kfan} The orientation of a $k$-fan if another $k'$-fan is added to $\{b,c\}$}
\end{figure}

As mentioned before edges with a vertex of in-degree $1$ and in-degree $2$ where the latter is a trouble maker may cause problems. We will anticipate these problems  using the following  procedure ( {\bf Procedure-1}), where $e$ is an edge and $x\in e$ is a trouble maker. Note that calling {\bf Procedure-1} often creates a new active $1$-$2$ edge and we deal with this by repeatedly calling {\bf Procedure-1} until there is no active $1$-$2$ edge stemming from the original edge $e$. \\[0.2cm]

\noindent {\bf {\bf Procedure-1}}$(e,x)$\\
{\bf If} no path gets attached to $e$ {\bf then} exit.\\
{\bf If} a $k$-fan gets attached to $e$ {\bf then} let $c$ be the vertex of the $5$-cycle of degree $2$ adjacent to $x$.\\
\hspace*{0.5cm} \parbox{12cm}{  
 {\bf If }$\{x,c\}$ will get attached a $k'$-fan {\bf then} use the  orientation of a $k$-fan as in Figure~\ref{added_kfan} and run {\bf Procedure-1}$(\{x,c\},x)$.\\
   {\bf otherwise} (no $k'$-fan) use the orientation of a $k$-fan as in Figure~\ref{normal_kfan} and run {\bf Procedure-1}$(\{x,c\},x).$ }\\
 {\bf otherwise} (no $k$-fan gets attached to $e$) let $P$ be the path that gets attached to $e$ and let $c$ be the vertex adjacent to $x$ in $P$. \\
 \hspace*{0.5cm} \parbox{12cm}{   {\bf If} a $k$-fan is attached to $\{x,c\}$ {\bf then} use Lemma~\ref{connectfan} to orient $P$ and call {\bf Procedure-1}$(\{x,c\},x)$  [Note that $c$ will have in-degree $3$ after {\bf Procedure-1} has finished and that $||P||\not=3$ as otherwise $P$ would be part of the fan.]\\
 {\bf otherwise} use Lemma~\ref{path} to orient $P$.
  Let $e'$ be the edge of $P$ incident to $x$. Call {\bf Procedure-1}$(e',x)$.
 }
\\[0.2cm]

Now we can describe our algorithm, which uses the following {\bf Procedure-2}  to avoid too many nested if-then statements. This procedure is only used in one particular case if a path $P$ of length $3$ is added to an edge $\{a,b\}$ with trouble maker $a$ and with $b$ having in-degree $3$.  In this case, if we use Lemma~\ref{path} we will create two active $1$-$2$ edges and therefore we have to be a bit more careful. The edge $e'$ will be the edge incident to $a$ on $P$ and we will assume that there is no $k$-fan on $e'$ when calling this procedure. The reader may want to come back to this procedure when it is called in the algorithm.\\[0.2cm]

\noindent{\bf Procedure-2}$(e',a,b)$\\
{\bf If} there will be a $k$-fan attached to the edge $e''$ of $P$ that is not incident to $a$ or $b$ {\bf then}

\hspace*{0.3cm} {\parbox{12cm}{   {\bf If}  it is a $0$-fan we use the following orientation:
\begin{center}
%
%
\includegraphics[scale=0.5]{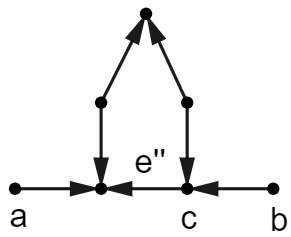}
\end{center}
 {\bf Otherwise} we use the following orientation for $P$ and the first path of length $3$ of the $k$-fan:
 
 \begin{center}
\includegraphics[scale=0.5]{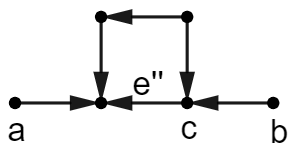}
\end{center}

 \mbox{}[Note that this does not create a new $1$-$2$ edge.]}\\
 
\noindent {\bf otherwise} use Lemma~\ref{path} to orient $P$ and call  {\bf Procedure-1}$(e',b)$ and then {\bf Procedure-1}$(e'',c)$ where $e''$ is the edge of $P$ that is not incident to $a$ or $b$ and $c$ is the vertex of this edge of in-degree $2$ (so adjacent to $b$)).\\[0.1cm]

\noindent {\bf Algorithm} \\
(A0) Choose an embedding of $G$ with all the vertices on the outer face. Choose a face and orient it using Lemma~\ref{path} identifying $v_1$ and $v_n$ (or orient it in any other proper way).\\
\noindent {\bf While} there exists an active edge $e$ we do the following:\\
Let $e=\{a,b\}$. W.l.o.g.\ we may assume that $d^{-}(a)<d^{-}(b)$. Let $P$ be the path that is attached to $e$.\\
\begin{enumerate}
\item[(A1)] {\bf If} neither $a$ nor $b$ is a trouble maker and $||P||\not=4$ {\bf then} orient $P$ as in Lemma~\ref{path}.
\item[(A2)] {\bf If} neither $a$ nor $b$ is a trouble maker and $||P||=4$ {\bf then} \\
\hspace*{0.3cm} \parbox{12cm}{ {\bf if} $d^{-}(a)=1$ and $d^{-}(b)=2$ {\bf then} use the following orientation (from $a$ to $b$):
\includegraphics[scale=0.25]{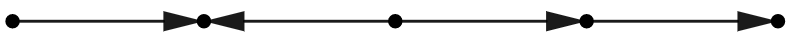}
\\
\mbox{} [Note that this introduces a vertex of in-degree $3$]\\
{\bf otherwise} [not a $1$-$2$ edge] use Lemma~\ref{path}.
}
\item[(A3)] {\bf If} $b$ is a trouble maker,  let $e'$ be the edge of $P$ incident to $b$\\
(We will see later that $d^{-}(a)\not=1.$)\\
\hspace*{0.3cm} \parbox{12cm}{ {\bf If} a $k$-fan will be attached to $e'$ we use Lemma~\ref{path} such that the in-degree $2$ vertex $c$ of $P$ is adjacent to $b$ (so $b=v_n$), and call {\bf Procedure-1}$(e',b)$. [Note that after we called {\bf Procedure-1} the in-degree of $b$ will be $2$ and the in-degree of the other vertex $c$ incident to $e'$ will be $3$.] }\\
{\bf Otherwise} we use Lemma~\ref{path} in the usual way and call {\bf Procedure-1}$(e',b)$.

\item[(A4)] {\bf If} $a$ is a trouble maker, let $e'$ be the edge on $P$ incident to $a$\\
\hspace*{0.3cm} {\parbox{12cm}{ {\bf If} a $k$-fan will be attached to $e'$ we use Lemma~\ref{connectfan23} with $v_1=a$ and call {\bf Procedure-1}$(e',a)$.\\
{\bf Otherwise} 
{\bf If} $||P||=3$ call {\bf Procedure-2}  $(e',a,b)$}\\  
\hspace*{2.4cm} {\bf otherwise} we use Lemma~\ref{connectfan23a} to orient $P$ and call\\ 
\hspace*{2.4cm} {\bf Procedure-1}$(e',a)$.
}
\end{enumerate}

\subsection{Correctness of the algorithm}

Let $D$ be an orientation of a subgraph $H$ of $G.$ A path $Q$ of length 2 in $H$ is a {\em $1$-$2$-$3$ path} (with respect to $H$) if the internal vertex of $Q$ is of in-degree $2$ and the other two vertices are of in-degrees $1$ and $3$ in $D$. 

\begin{lemma}\label{first}
The only way to get an active edge $e=\{a,b\}$ such that $a$ is a trouble maker and $d^{-}(b)=1$ is to attach a path or a $k$-fan to an edge incident to a trouble maker. 
\end{lemma}
\begin{proof}
We first look at ways to create vertices of in-degree $3$. In step (A1) we cannot create a new vertex of in-degree $3$. In step (A2) we create a vertex $x$ of in-degree $3$, but it cannot have a neighbour of in-degree $2$ as we started with a proper orientation and $x$ had in-degree $2$ before we attached the path of length $4$ and the in-degree of no other vertex changes. Therefore we cannot create a $1$-$2$-$3$ path in this step. 
\end{proof}

\begin{lemma}
At each iteration of the main {\bf while}-loop there is no active $1$-$2$ edge of a $1$-$2$-$3$ path.
\end{lemma}

\begin{proof}
We have seen in Lemma~\ref{first} that we cannot get an active $1$-$2$ edge of a $1$-$2$-$3$ path in (A1) and (A2). So now assume the lemma is true as we enter (A3). If a $k$-fan is
attached, then the neighbour $c$ of $b$ in $P$ will have in-degree $3$ and the neighbour $d$ of $c$ in $P$ will have in-degree at most $1$. So $c$ does not create a new trouble 
maker and no new $1$-$2$-$3$ path. We then may repeatedly call Procedure-1. Note that we may create an active $1$-$2$ edge in a $1$-$2$-$3$ path but we immediately will call Procedure-1 on these edges. 
We only finish when no more paths or $k$-fans are attached and so the $1$-$2$ edge is inactive.

We can argue similarly for (A4) but here we have some added complications because $d^{-}(b)=3$ and we do not want to introduce a new $1$-$2$-$3$ path. We do so by insuring that 
either the vertex adjacent to $b$ in $P$ has in-degree $1$ (if $P$ is even, $||P||\not=4$) or has in-degree $2$ but its neighbour has in-degree equal to $0$ ($P$ odd, $||P||=4$). 
Also we may add the path of length $3$ in which case we get two active $1$-$2$ edges in a $1$-$2$-$3$ path if we use Lemma~\ref{path}. Note, that adding a path of length $3$ to a $2$-$3$ edge (i.e. an edge with vertex in-degrees $2$ and $3$) is the only way to create more than one active $1$-$2$ edge in a $1$-$2$-$3$ path.

\begin{figure}
\begin{center}
%
\includegraphics[scale=0.5]{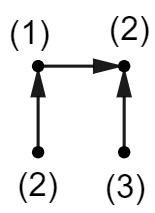}
\end{center}
\caption{\label{2313} The required orientation of the path. The in-degrees are in brackets.}
\end{figure}

We deal with this case by first considering the $1$-$2$-$3$ path that contains only one edge of the new path. If a $k$-fan is added then we choose a different orientation of the path (see Figure \ref{2313}), which yields the in-degree sequence $2$, $3$, $1$, $3$, where the degree $2$ vertex is (and remains) the trouble maker. In this case, we do not have to deal with a second active $1$-$2$ edge. We may create a new active $1$-$2$ edge which is part of the $k$-fan but we repeatedly use Procedure-1 until no more paths or $k$-fans get added to such an edge.

If a $k$-fan is added to the edge not containing the trouble maker then we also change the orientation of the path and either add a path of length $4$ (in case of a $0$-fan) or one path of length $3$ of the $k$-fan. In either case, we do not create a new $1$-$2$ edge and do not have to use Procedure-1.

If no $k$ fan is added to either of these edges then we just use Lemma~\ref{path} and then Procedure-1 on both edges until no active $1$-$2$ edge is present.
\end{proof}

It is now straightforward to check that using our algorithm we create a proper 3-orientation of $G$ at the end of each step (A1), (A2), (A3) and (A4). Thus Theorem \ref{main2} follows.

\subsection{A $2$-connected triangle-free outerplanar graph $G$ with $\vec{\chi}(G)=3$ }

Consider the $2$-connected outerplanar graph consisting of one cycle $C_5$ of length $5$ and a path of length $4$ attached to each of its edges, see Figure~\ref{example}.

\begin{figure}[h]
\begin{center}
%
%
%
%
%
%
%
%
%
%
%
%
%
%
%
\includegraphics[scale=0.45]{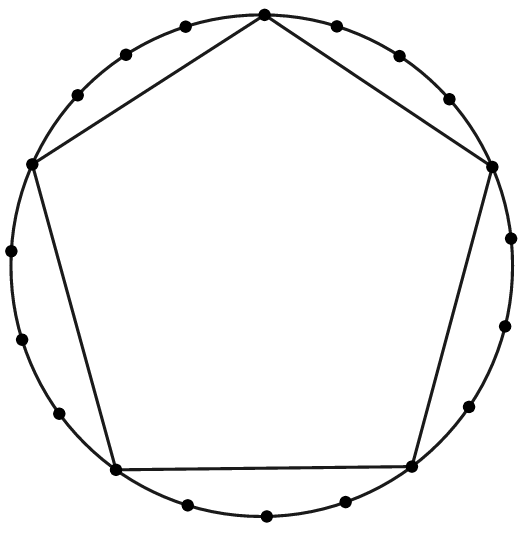}
\end{center}
\caption{\label{example} An example of a graph $G$ with $\vec{\chi}(G)=3$}
\end{figure}

Assume for a contradiction that we can orient the graph in such a way that the maximum in-degree is less than $3$. The only way to orient the $5$-cycle without creating a $1$-$2$-edge is to orient the edges cyclicly so that every vertex in $C_5$ has in-degree equal to $1$.
 Note that a $1$-$2$-edge $e$ in $C_5$ would mean that we cannot orient the path of length $4$ attached to $e$ in a proper way without increasing the in-degree of one of the endpoints of $e$; and since the orientation should remain proper that means that the in-degree of one vertex has to become $3$. But if all the in-degrees on $C_5$ are equal to one, then we do not have a proper orientation, so when we attach the first path of length $4$ to an edge $e$, one of the vertices $a$ of $e$ will get an in-degree of $2$. Observe that the other edge incident to $a$ in $C_5$ is still active and this forces the orientation to give $a$ in-degree $3$ as above.

\section{Proofs of Theorems \ref{bridgefree} and \ref{treefree}}\label{sec:bridgefree}
A {\em block} in a graph $G$ is a maximal 2-connected subgraph of $G.$
A \emph{block tree} of a connected graph $G$ denoted by $B(G)$ is a bipartite graph with bipartition $(B, C)$, where $B$ is the set of blocks of $G$ and $C$ is the set of cut-vertices of $G$, a block $L$ and a cut-vertex $v$ being adjacent in $B(G)$ if and only if $L$ contains $v$. It is not hard to see that $B(G)$ is a tree \cite{BM}.

To prove Theorem \ref{bridgefree}, it suffices to consider a connected, bridgeless, outerplanar graph $G$. Choose a block $L_1$ of $G$ as a root of $B(G)$ and apply breadth-first search on $B(G)$ from $L_1$ to visit all blocks of $G$ one by one. Now if we consider the blocks of $G$ in the order $L_1,\dots ,L_p$ they were visited,
we can orient the edges of one block and then extend it to another block sharing a cut-vertex as follows, without ever encountering a block we have already oriented. 

Let $G'$ be an outerplanar embedding of $G$ into the plane. Consider $L_1$ and orient its edges as in Section~\ref{2connected}. Assume that we have oriented the edges of blocks $L_1,\dots ,L_{j-1}$ and we wish to orient edges of $L_j.$ Assume that  a cut-vertex $v$ is the parent of $L_j$ on the rooted tree $B(G)$ and let $H$ be the orientation of $L_1\cup \dots \cup L_{j-1}.$
If $d^-_H(v)\neq 2$, then we choose a face of $G'$ in $L_j$ containing $v$ and use Lemma~\ref{connectfan}, where we identify $v_1$ and $v_n$ with $v,$ to get the orientation of the border cycle of $L_j$ in $G'.$
If $d_H^-(v)=2$, we just orient the edges towards $v$ as in Lemma~\ref{in1} so that $v$ will have in-degree $4$. Note that we only change an in-degree $2$ vertex to an in-degree $4$ vertex, so no in-degree $4$ vertices will ever be adjacent.  
Now we can just proceed orienting $L_j$ with our algorithm for the $2$-connected case. 

Now consider a connected, triangle-free, tree-free, outerplanar graph $G$. We proceed similarly to the bridgeless case, but now sometimes two blocks are connected by a bridge instead of a cut-vertex. So assume we have oriented the edges in $L_1,\ldots, L_i$ and we want to orient the edges of $L_{i+1}$ connected to $L_1,\ldots, L_i$ via a bridge $e=\{a,b\}$ where $b\in L_{i+1}$. We orient $e$ towards $b$. If the in-degree of $a$ in the orientation $F$ of $L_1\cup \dots \cup L_i$
does not equal $1$, then we consider a face of $G'$ containing $b$ in $L_{i+1}$ and orient its border cycle using Lemma~\ref{connectfan} identifying $v_1$ and $v_n$ with $b$. If the in-degree of $a$ in $F$ does equal $1$, then we use Lemma~\ref{in1} to obtain a vertex of in-degree $3$. Note that the neighbours of $b$ in $L_{i+1}$ have in-degree $0$ in $F$, so we do not create a $1$-$2$-$3$ path.
We can now proceed to orient the edges in $L_{i+1}$ as in Section~\ref{2connected}.

\section{Open Problems}\label{sec:prob}

We do not know whether the bounds of Theorems \ref{bridgefree} and \ref{treefree} are tight or not. However, this is not important in attacking the open question of whether there is a constant $c$ such that 
$\vec{\chi}(G)\leq c$ for every outerplanar graph $G$. Thus, it would be more interesting to obtain a non-trivial extension of Theorem  \ref{treefree} by either removing the condition that $G$ is triangle-free or that $G$ is tree-free. Another way of attacking the open question is to extend the cactus bound of \cite{AJ} from cacti to a wider class of graphs. 

\vspace{3mm}

\noindent {\bf Acknowledgements:} 
We would like to thank 
the referees for their careful reading and helpful comments.
Shi and Taoqiu were partially supported by the National Natural Science Foundation of China (No. 11922112).

\end{document}